\documentclass[preprint,12pt]{elsarticle}
\usepackage{amsfonts}
\usepackage{amssymb}
\usepackage{graphicx}
\usepackage{amsthm}
\usepackage{lmodern}
\usepackage{hyperref}
\usepackage{caption}
\usepackage{systeme}
\usepackage{float}         
\usepackage{caption}       
\usepackage{subcaption}    

\usepackage{epstopdf}
\usepackage{algorithmic}
\newtheorem{theorem} {Theorem}[section]
\newtheorem{proposition}  [theorem]{Proposition}
\newtheorem{definition}    [theorem]{Definition}
\newtheorem{lemma}[theorem]{Lemma}
\newtheorem{remark}[theorem]{Remark}
\newtheorem{assumption}[theorem]{Assumption}
\newtheorem{corollary}[theorem]{Corollary}

\usepackage{amssymb}
\usepackage{amsmath}

\begin{document}

\begin{frontmatter}



\title{Energy Decay Induced by Localized Coercivity in the 3D Navier--Stokes Equations}

\author[]{Amadou Cissé}\ead{amadou.cisse@univ-lorraine.fr}

\affiliation{organization={CRAN-UMR-CNRS-7039 University of Lorraine, F-57000, France}
             }


\begin{abstract}
Energy-based analysis of the three-dimensional incompressible Navier--Stokes
equations with spatially localized dissipation is developed at the level of
kinetic energy and Leray--Hopf weak solutions.
A localized coercivity property linking viscosity and damping to global $L^2$
control is introduced and related to observability and spectral inequalities for
the associated Stokes operator.
Under this condition, exponential energy decay is established without regularity
or linearization assumptions, and time-dependent damping is treated within a
nonautonomous framework.
A finite-dimensional spectral quantity is proposed as a computable proxy for
coercivity, yielding explicit decay estimates for reduced-order models.
\end{abstract}




\end{frontmatter}

\section{Introduction}

The incompressible Navier--Stokes equations constitute a fundamental mathematical
model for viscous fluid flows and play a central role in hydrodynamics,
aerodynamics, and turbulence modeling.
In three space dimensions, these equations exhibit a wide range of complex
phenomena, including instabilities, transition to turbulence, and chaotic
dynamics.
From a mathematical standpoint, the question of global existence and uniqueness
of smooth solutions remains one of the major open problems in analysis
\cite{Fefferman2006,TemamNS2001,FoiasTemamTurbulence2001}.

\medskip

Let $\mathbb{T}^3=(\mathbb{R}/\mathbb{Z})^3$ denote the three-dimensional torus,
which models a periodic flow domain without physical boundaries.
This setting eliminates boundary effects and allows one to isolate the role of
internal dissipation and volumic actuation. The incompressible Navier–Stokes equations with localized volumic damping read
\begin{equation}
\label{eq:NS_controlled}
\begin{cases}
\partial_t u + (u\cdot\nabla)u - \nu \Delta u + \nabla p
= - \chi_\omega(x)\, K(\rho(t))\, u,\\[1mm]
\nabla\cdot u = 0,\\
u(\cdot,0)=u_0,
\end{cases}
\end{equation}
where $u:\mathbb{T}^3\times\mathbb{R}_+\to\mathbb{R}^3$ denotes the velocity field,
$p$ is the pressure enforcing incompressibility, and $\nu>0$ is the kinematic
viscosity.

The right-hand side represents a \emph{localized volumic dissipation} acting on
a measurable subset $\omega\subset\mathbb{T}^3$.
The function $\chi_\omega\in L^\infty(\mathbb{T}^3)$ is assumed nonnegative and
typically denotes either the indicator function of $\omega$ or a smooth cutoff
supported in $\omega$.
From a physical viewpoint, this term may be interpreted as an idealized model of
distributed mechanisms extracting kinetic energy from the flow in a prescribed
region.
It is not intended to represent a realizable control input with finitely many
actuators, but rather a continuous volumic dissipation model used to analyze
the effect of spatially localized energy extraction on the global kinetic
energy balance.

The scalar gain $K(\rho(t))\ge 0$ is allowed to vary in time and depends on a
scheduling variable $\rho(t)$, assumed measurable and bounded.

\medskip

These intrinsic analytical difficulties have direct consequences for stability
and control.
In particular, the design of feedback laws ensuring global stabilization of
three-dimensional Navier--Stokes flows in strong norms is currently out of reach.
As a result, rigorous results are typically restricted to two-dimensional
models, to local stabilization around sufficiently regular solutions, or to
stability notions formulated at the level of weak solutions; see, for instance,
\cite{Barbu2003,BadraTakahashi2009a,BadraTakahashi2009b,Raymond2006}.
This limitation reflects the absence of a global regularity theory in three
dimensions rather than a purely technical obstruction.

Early contributions to stabilization and control of the Navier--Stokes equations
focused on open-loop or feedback strategies around steady or time-dependent
reference flows, often under smallness assumptions on the initial data or on the
perturbations.
Representative works include those of Barbu \cite{Barbu2003}, Badra and Takahashi
\cite{BadraTakahashi2009a}, and boundary control approaches based on linearization
and spectral analysis \cite{Raymond2006}.
While these results provide important insight into stabilization mechanisms,
their applicability is generally confined to regimes where sufficient solution
regularity can be guaranteed.

A complementary and analytically robust perspective relies on energy methods and
dissipativity arguments.
For incompressible flows, the nonlinear convection term is skew-symmetric with
respect to the $L^2$ inner product and therefore does not contribute to the
kinetic energy balance.
This structural property remains valid for Leray--Hopf weak solutions and enables
stability analyses that do not rely on strong regularity assumptions
\cite{TemamNS2001,FoiasTemamTurbulence2001}.
It has motivated the introduction of damping mechanisms aimed at enhancing the
natural viscous dissipation of the system.
Global damping terms are known to improve well-posedness and decay properties in
several three-dimensional settings, although such mechanisms may be difficult to
justify from a physical or modeling standpoint
\cite{cai2008weak,zhou2012regularity,vasseur2016global,raugel1993navier}.

In realistic configurations, dissipation or actuation mechanisms are often
spatially localized.
This naturally leads to the study of Navier--Stokes equations with localized
volumic or boundary damping.
In this setting, a central analytical difficulty is to relate localized
dissipation to global energy decay.
This issue is closely related to \emph{coercivity} and \emph{observability}
properties, and more generally to unique continuation and spectral inequalities.

On periodic domains such as $\mathbb{T}^3$, the Stokes operator is already
exponentially stable on divergence-free zero-mean fields due to the Poincar\'e
inequality.
As a consequence, the introduction of a localized damping term does not aim at
stabilizing an unstable system, but rather at modifying and enhancing the
underlying energy dissipation mechanism.
The relevant question is therefore not whether decay occurs, but how spatially
localized dissipation contributes quantitatively and structurally to the global
energy balance, and how this contribution can be characterized within the
framework of weak solutions.

For linear parabolic equations, such questions are well understood and are
classically addressed through observability inequalities and Carleman estimates;
see, for instance, the seminal work of Lebeau and Robbiano
\cite{LebeauRobbiano1995} and subsequent developments
\cite{apraiz2014observability}.
Related spectral inequalities quantify the extent to which localized dissipation
controls global norms.

In the context of fluid equations, observability and controllability properties
have also been extensively studied for linearized Navier--Stokes systems,
in particular for Oseen--Stokes type equations.
These results rely on refined Carleman estimates and unique continuation
arguments; see, for instance,
\cite{fursikov1996imanuvilov,fursikov1996local,coron1996global,coron2009null}.
They provide a detailed understanding of observability mechanisms for linearized
fluid models under suitable regularity and geometric assumptions.

In contrast, for nonlinear fluid models such as the Navier--Stokes equations,
the transfer of localized dissipation to global energy control remains delicate
and only partially understood; see, for example, \cite{Shirikyan2007}.
In particular, no global observability inequality at the energy level is
currently available for the three-dimensional Navier--Stokes equations in the
Leray--Hopf framework.

Parallel to these developments, operator-theoretic approaches to
nonautonomous evolution equations provide natural tools for analyzing
systems with coefficients depending on time.
Semigroup and evolution-family techniques \cite{Kato1970,Pazy1983}
allow the treatment of parameter-dependent generators and the derivation
of stability estimates under dissipativity and coercivity assumptions.
These methods are particularly suited to dissipation mechanisms with
time-varying intensity, yet their systematic use in the analysis of
three-dimensional Navier--Stokes equations remains limited.

Recent years have seen renewed interest in stabilization of the
three-dimensional Navier--Stokes equations under localized dissipation or
finite-dimensional feedback mechanisms, particularly in weak or critical
functional settings.
Uniform stabilization results in Sobolev and Besov spaces have been obtained in
\cite{LasieckaPiyasadTriggiani2019,LasieckaPriyasadTriggiani2021}, relying on
refined observability and compactness arguments.
Related approaches based on optimization and receding-horizon methods have been
considered in \cite{Azmi2022}.
These contributions further emphasize the central role of localized coercivity
and observability in stabilization problems.

\medskip

The present work adopts an energy-based analytical approach to the
$3D$ incompressible Navier--Stokes equations with localized
volumic dissipation.
Rather than addressing stabilization in strong norms, the analysis is conducted
in the kinetic energy space $L^2$, which is compatible with the Leray--Hopf
existence theory.
The damping mechanism is modeled as a spatially localized dissipative term with
time-dependent intensity, leading to a family of parameter-dependent linear
operators governing the dissipative component of the dynamics.

A first contribution is the formulation of a localized coercivity property that
provides a quantitative link between viscous diffusion, spatially localized
dissipation, and global $L^2$ energy control.
This formulation is intrinsic to the energy structure of Leray--Hopf solutions
and extends classical coercivity concepts to a nonautonomous dissipative setting.

A second contribution is the introduction of a finite-dimensional spectral
quantity that provides a computable characterization of the coercivity mechanism
on divergence-free reduced spaces.
This construction yields explicit decay estimates and offers a practical
framework for assessing the effectiveness of localized damping in
high-dimensional flow models.
No convergence claim toward the infinite-dimensional problem is made; the
procedure is intended as a consistent proxy for the underlying coercivity
mechanism.

\medskip

In summary, this paper provides an energy-level characterization of how localized
dissipation contributes to global energy decay for the three-dimensional
Navier--Stokes equations through coercivity and observability mechanisms,
without relying on global regularity or linearization.

\medskip

The remainder of the paper is organized as follows.
Section~\ref{sec:energy_stabilization} develops the energy-based framework for
the Navier--Stokes system with localized dissipation and establishes the
fundamental energy inequality at the level of Leray--Hopf weak solutions.
Section~\ref{sec:operator_kato} develops an operator-theoretic formulation of the
dynamics and derives stability estimates for the associated nonautonomous linear
evolution family.
Section~\ref{sec:continuous_observability} discusses continuous observability and
localized coercivity properties, clarifying their role in linking localized
dissipation to global energy decay.
Section~\ref{sec:discrete_coercivity} presents a finite-dimensional spectral
procedure and interprets it as a computable proxy for the continuous coercivity
mechanism.
Concluding remarks and perspectives are given in the final section.

\section{Energy-Based Analysis of Navier--Stokes with Localized Dissipation}
\label{sec:energy_stabilization}

This section establishes the energy-based framework underlying the analysis of
the three-dimensional incompressible Navier--Stokes equations with localized
volumic dissipation.
The focus is on the influence of spatially localized dissipation mechanisms on
the kinetic energy, in a setting compatible with Leray--Hopf weak solutions.

The functional framework relies on the classical divergence-free spaces
\[
H := \Big\{ v\in L^2(\mathbb{T}^3;\mathbb{R}^3)\ :\ \nabla\cdot v = 0,\ 
\int_{\mathbb{T}^3} v\,dx = 0 \Big\},\qquad
V := H^1(\mathbb{T}^3)\cap H,
\]
endowed with their natural norms.
The zero-mean condition ensures the validity of the Poincar\'e inequality on the
periodic domain.

\subsection{Weak solutions and energy balance}

For any initial condition $u_0\in H$, and any bounded measurable gain
$K(\rho(\cdot))$, system \eqref{eq:NS_controlled} admits at least one global
weak solution in the sense of Leray--Hopf.

The analysis relies on the energy balance satisfied by such solutions.

\begin{lemma}[Energy inequality]
\label{lem:energy_ineq}
Let $u$ be a Leray--Hopf weak solution of \eqref{eq:NS_controlled}.
Then, for almost every $t>0$,
\begin{equation}
\label{eq:energy_ineq}
\frac{1}{2}\frac{d}{dt}\|u(t)\|_{L^2}^2
+ \nu \|\nabla u(t)\|_{L^2}^2
+ \int_{\omega} K(\rho(t))\,|u(x,t)|^2\,dx
\le 0.
\end{equation}
Equivalently, for almost every $s\ge 0$ and every $t\ge s$,
\begin{equation}
\label{eq:energy_ineq_integrated}
\frac{1}{2}\|u(t)\|_{L^2}^2
+ \nu \int_s^t \|\nabla u(\tau)\|_{L^2}^2\,d\tau
+ \int_s^t \int_\omega K(\rho(\tau))\,|u(x,\tau)|^2\,dx\,d\tau
\le \frac{1}{2}\|u(s)\|_{L^2}^2.
\end{equation}
\end{lemma}

\begin{proof}
The estimate is first established at the Galerkin level.
Let $(w_k)_{k\ge1}\subset V$ be a divergence-free orthonormal basis of $H$ and
consider an approximate solution of the form
\[
u_m(t)=\sum_{k=1}^m g_k^{(m)}(t)w_k.
\]
Projecting \eqref{eq:NS_controlled} onto
$\mathrm{span}\{w_1,\dots,w_m\}$ yields a finite-dimensional system for the
coefficients $g_k^{(m)}$.
Testing this system with $u_m(t)$ gives
\[
\frac12\frac{d}{dt}\|u_m(t)\|_{L^2}^2
+\nu\|\nabla u_m(t)\|_{L^2}^2
+\int_\omega K(\rho(t))|u_m(x,t)|^2\,dx
=
-\int_{\mathbb{T}^3}(u_m\cdot\nabla)u_m\cdot u_m\,dx .
\]

The convection term vanishes due to incompressibility and periodicity:
\[
\int_{\mathbb{T}^3}(u_m\cdot\nabla)u_m\cdot u_m\,dx
=
\frac12\int_{\mathbb{T}^3}u_m\cdot\nabla(|u_m|^2)\,dx
=0.
\]
Therefore,
\[
\frac12\frac{d}{dt}\|u_m(t)\|_{L^2}^2
+\nu\|\nabla u_m(t)\|_{L^2}^2
+\int_\omega K(\rho(t))|u_m(x,t)|^2\,dx
=0.
\]
Integrating over $(s,t)$ yields the exact energy identity at the Galerkin level.

The standard compactness argument in the Leray construction applies here,
since the additional term
$-\chi_\omega K(\rho(t))u_m$ is linear with bounded measurable coefficient.
Passing to the limit as $m\to\infty$ yields a Leray--Hopf weak solution $u$ and
the integrated inequality \eqref{eq:energy_ineq_integrated}.
The differential form \eqref{eq:energy_ineq} follows for almost every $t>0$.
\end{proof}

As a direct consequence, the kinetic energy $\|u(t)\|_{L^2}^2$ is
non-increasing in time for any admissible gain $K(\rho(t))\ge 0$.
Lemma~\ref{lem:energy_ineq} provides the fundamental estimate used throughout
the paper.

\subsection{Exponential decay under a localized coercivity condition}

The energy inequality alone does not yield a uniform decay rate.
Exponential decay follows from an additional coercivity mechanism linking
localized dissipation to global kinetic energy.

\begin{proposition}[Localized coercivity]
\label{prop:coercivity_optimal}
Let $\chi_\omega\in L^\infty(\mathbb{T}^3)$ be nonnegative and let
$K_{\min}\ge 0$.
Define
\begin{equation}
\label{eq:alphaomega_def}
\alpha_\omega(K_{\min})
:=
\inf_{v\in V\setminus\{0\}}
\frac{
\nu \|\nabla v\|_{L^2(\mathbb{T}^3)}^2
+
K_{\min}\|\chi_\omega^{1/2} v\|_{L^2(\mathbb{T}^3)}^2
}{
\|v\|_{L^2(\mathbb{T}^3)}^2
}.
\end{equation}
Then, for all $v\in V$,
\begin{equation}
\label{eq:coercivity}
\nu \|\nabla v\|_{L^2(\mathbb{T}^3)}^2
+
K_{\min}\|\chi_\omega^{1/2} v\|_{L^2(\mathbb{T}^3)}^2
\ge
\alpha_\omega(K_{\min}) \|v\|_{L^2(\mathbb{T}^3)}^2.
\end{equation}
Moreover, $\alpha_\omega(K_{\min})>0$ if and only if a coercivity inequality of
the form \eqref{eq:coercivity} holds, and in that case
$\alpha_\omega(K_{\min})$ is the largest admissible constant.
\end{proposition}

\begin{proof}
For $v\in V\setminus\{0\}$, define the Rayleigh quotient
\[
R(v):=
\frac{
\nu \|\nabla v\|_{L^2(\mathbb{T}^3)}^2
+
K_{\min}\|\chi_\omega^{1/2} v\|_{L^2(\mathbb{T}^3)}^2
}{
\|v\|_{L^2(\mathbb{T}^3)}^2
}.
\]
By definition,
\[
\alpha_\omega(K_{\min})=\inf_{v\in V\setminus\{0\}} R(v).
\]
Hence, for every $v\in V\setminus\{0\}$,
\[
\alpha_\omega(K_{\min})\le R(v),
\]
which is exactly \eqref{eq:coercivity} after multiplication by
$\|v\|_{L^2(\mathbb{T}^3)}^2$.
The case $v=0$ is trivial.

Suppose now that there exists $\alpha>0$ such that
\[
\nu \|\nabla v\|_{L^2}^2
+
K_{\min}\|\chi_\omega^{1/2} v\|_{L^2}^2
\ge
\alpha \|v\|_{L^2}^2,
\qquad \forall v\in V.
\]
Then $R(v)\ge \alpha$ for every $v\neq0$, and taking the infimum over
$V\setminus\{0\}$ yields
\[
\alpha_\omega(K_{\min})\ge \alpha>0.
\]

Conversely, if $\alpha_\omega(K_{\min})>0$, then \eqref{eq:coercivity} holds
with $\alpha=\alpha_\omega(K_{\min})$ by the first part of the proof.

Finally, let $\beta$ be any constant such that
\[
\nu \|\nabla v\|_{L^2}^2
+
K_{\min}\|\chi_\omega^{1/2} v\|_{L^2}^2
\ge
\beta \|v\|_{L^2}^2,
\qquad \forall v\in V.
\]
Then $R(v)\ge \beta$ for all $v\neq0$, hence
$\alpha_\omega(K_{\min})\ge \beta$.
Therefore $\alpha_\omega(K_{\min})$ is the largest admissible constant.
\end{proof}

\begin{remark}[Interpretation]
The quantity $\alpha_\omega(K_{\min})$ corresponds to the spectral gap of the
quadratic form associated with the operator
$-\nu\Delta + K_{\min}\chi_\omega$ on divergence-free, zero-mean fields.
Its positivity reflects the ability of localized dissipation to control the
global kinetic energy.
\end{remark}

\begin{corollary}[A sufficient non-localized condition]
\label{cor:coercivity_sufficient}
Assume that $\chi_\omega(x)\ge \chi_{\min}>0$ almost everywhere on
$\mathbb{T}^3$.
Then
\[
\alpha_\omega(K_{\min})\ge \nu\lambda_1 + K_{\min}\chi_{\min}>0,
\]
where $\lambda_1>0$ denotes the Poincar\'e constant on $H$.
\end{corollary}

\begin{proof}
For every $v\in V$, the lower bound on $\chi_\omega$ gives
\[
\|\chi_\omega^{1/2}v\|_{L^2}^2
=
\int_{\mathbb{T}^3}\chi_\omega |v|^2\,dx
\ge
\chi_{\min}\|v\|_{L^2}^2.
\]
Moreover, by the Poincar\'e inequality on $H$,
\[
\|\nabla v\|_{L^2}^2\ge \lambda_1\|v\|_{L^2}^2.
\]
Therefore
\[
\nu \|\nabla v\|_{L^2}^2
+
K_{\min}\|\chi_\omega^{1/2} v\|_{L^2}^2
\ge
(\nu\lambda_1+K_{\min}\chi_{\min})\|v\|_{L^2}^2.
\]
Taking the infimum over $v\in V\setminus\{0\}$ in
\eqref{eq:alphaomega_def} yields the desired bound.
\end{proof}

\begin{remark}
The condition $\chi_\omega\ge \chi_{\min}>0$ corresponds to damping distributed
throughout the whole domain and is therefore restrictive in genuinely localized
settings.
It is included only as an analytical benchmark.
\end{remark}

A lower bound is imposed on the time-varying gain:
\begin{equation}
\label{eq:K_lower_bound}
K(\rho(t)) \ge K_{\min}
\quad \text{for almost every } t\ge 0.
\end{equation}

This lower bound ensures that the coercivity mechanism quantified by
$\alpha_\omega(K_{\min})$ remains uniformly available along the evolution.

\begin{theorem}[Exponential energy decay]
\label{thm:exp_decay}
Assume that $\alpha_\omega(K_{\min})>0$ and that
\eqref{eq:K_lower_bound} holds.
Then any Leray--Hopf weak solution of \eqref{eq:NS_controlled} satisfies
\begin{equation}
\label{eq:exp_decay}
\|u(t)\|_{L^2} \le e^{-\alpha_\omega(K_{\min}) t}\|u_0\|_{L^2},
\qquad \forall t\ge 0.
\end{equation}
\end{theorem}

\begin{proof}
Let $u$ be a Leray--Hopf weak solution.
By Lemma~\ref{lem:energy_ineq}, for almost every $s\ge0$ and every $t\ge s$,
\[
\frac12\|u(t)\|_{L^2}^2
+\nu\int_s^t\|\nabla u(\tau)\|_{L^2}^2\,d\tau
+\int_s^t\int_\omega K(\rho(\tau))|u(x,\tau)|^2\,dx\,d\tau
\le
\frac12\|u(s)\|_{L^2}^2.
\]
Since
\[
\int_\omega K(\rho(\tau))|u(x,\tau)|^2\,dx
=
K(\rho(\tau))\|\chi_\omega^{1/2}u(\tau)\|_{L^2}^2,
\]
the lower bound \eqref{eq:K_lower_bound} gives, for almost every $\tau$,
\[
\nu\|\nabla u(\tau)\|_{L^2}^2
+K(\rho(\tau))\|\chi_\omega^{1/2}u(\tau)\|_{L^2}^2
\ge
\nu\|\nabla u(\tau)\|_{L^2}^2
+K_{\min}\|\chi_\omega^{1/2}u(\tau)\|_{L^2}^2.
\]
Applying \eqref{eq:coercivity} with $v=u(\tau)$ yields
\[
\nu\|\nabla u(\tau)\|_{L^2}^2
+K(\rho(\tau))\|\chi_\omega^{1/2}u(\tau)\|_{L^2}^2
\ge
\alpha_\omega(K_{\min})\|u(\tau)\|_{L^2}^2
\]
for almost every $\tau$.
Substituting into the integrated energy inequality gives
\[
\frac12\|u(t)\|_{L^2}^2
+\alpha_\omega(K_{\min})\int_s^t \|u(\tau)\|_{L^2}^2\,d\tau
\le
\frac12\|u(s)\|_{L^2}^2.
\]
Define $E(t):=\|u(t)\|_{L^2}^2$.
Then, for almost every $s\ge0$ and every $t\ge s$,
\[
E(t)+2\alpha_\omega(K_{\min})\int_s^t E(\tau)\,d\tau \le E(s).
\]
A standard integral form of Gr\"onwall's lemma implies
\[
E(t)\le e^{-2\alpha_\omega(K_{\min})(t-s)}E(s),
\qquad \forall t\ge s,
\]
for almost every $s\ge0$.
Letting $s\to0^+$ and using the standard strong convergence
$u(s)\to u_0$ in $L^2$ for Leray--Hopf weak solutions, one obtains
\[
E(t)\le e^{-2\alpha_\omega(K_{\min})t}E(0)
=
e^{-2\alpha_\omega(K_{\min})t}\|u_0\|_{L^2}^2,
\qquad \forall t\ge0,
\]
which is equivalent to \eqref{eq:exp_decay}.
\end{proof}

\begin{remark}[Role of localized dissipation]
On the periodic domain $\mathbb{T}^3$, exponential decay already follows from
viscous dissipation alone through the Poincar\'e inequality.
The role of the localized damping term is therefore not to stabilize an unstable
system, but to introduce a coercivity mechanism linking spatially localized
dissipation to global energy decay.
In particular, it makes it possible to quantify how the geometry of the damping
region influences the decay rate through the constant
$\alpha_\omega(K_{\min})$.
\end{remark}

\subsection{Discussion}

\begin{itemize}
\item The analysis is carried out at the level of Leray--Hopf weak solutions.
\item Time-varying gains are allowed, provided \eqref{eq:K_lower_bound} holds
when exponential decay is required.
\item The skew-symmetric structure of the convection term isolates the role of
dissipation in the energy balance.
\end{itemize}

\begin{remark}[Localized damping structure]
\label{rem:physical_damping}
The term $-\chi_\omega K(\rho(t))u$ is chosen in a proportional and dissipative
form.
From a physical standpoint, it models resistive mechanisms that locally extract
momentum.
From a mathematical viewpoint, it preserves the skew-symmetric energy balance of
the convective nonlinearity and contributes an additional localized dissipation
term in the energy identity.
\end{remark}

%
%
%
%
%

\subsection{On the nature and verification of the coercivity condition}

The positivity of the constant $\alpha_\omega(K_{\min})$ in
\eqref{eq:alphaomega_def} is a structural property of the pair
$(\chi_\omega,K_{\min})$.
It reflects the ability of the combined viscous and localized dissipation
mechanisms to control the global $L^2$ energy.

In general, this property depends on global spectral features of the operator
$-\nu\Delta + K_{\min}\chi_\omega$ restricted to divergence-free, zero-mean
fields and cannot be verified by purely local arguments.
It is closely related to observability and spectral properties for parabolic
operators; see Section~\ref{sec:continuous_observability}.

Two complementary viewpoints can be distinguished:

\begin{itemize}
\item[(i)] \emph{Analytical verification.}
Under suitable geometric assumptions on the damping region $\omega$,
the coercivity condition is expected to be related to observability or
spectral inequalities for parabolic equations; see, for instance,
\cite{LebeauRobbiano1995,apraiz2014observability}.
However, establishing such properties in the present setting remains a
nontrivial issue and is not addressed here.

\item[(ii)] \emph{Finite-dimensional verification.}
In the absence of explicit analytical criteria, the coercivity condition can be
assessed on divergence-free reduced spaces.
This leads to the discrete spectral procedure introduced in
Section~\ref{sec:discrete_coercivity}, which provides a computable proxy for the
continuous coercivity mechanism at a prescribed resolution.
\end{itemize}

In the present work, the positivity of $\alpha_\omega(K_{\min})$ is treated as
a structural assumption at the continuous level.
The purpose of the analysis is to characterize its consequences at the level of
energy estimates rather than to derive explicit sufficient conditions.

The next section develops an operator-theoretic formulation of the dissipative
part of the dynamics.

\section{Operator Estimates for Time-Dependent Localized Dissipation}
\label{sec:operator_kato}

This section recasts the dissipative part of the Navier--Stokes dynamics with
localized dissipation into an operator-theoretic setting.
The time dependence enters only through the localized damping intensity
$K(\rho(t))$, while the underlying Stokes dynamics remains autonomous.
The purpose is to make this dissipative structure explicit and to analyze the
effect of time-dependent damping through a nonautonomous evolution setting.

\subsection{Stokes operator and Leray projector}

Let $P:L^2(\mathbb{T}^3;\mathbb{R}^3)\to H$ denote the Leray projector onto
divergence-free, zero-mean vector fields.
The periodic Stokes operator is defined by
\[
A_0 := \nu P\Delta,
\qquad
D(A_0):= H^2(\mathbb{T}^3;\mathbb{R}^3)\cap H.
\]
Since velocity fields in $H$ are divergence-free and have zero spatial mean, the
Laplacian preserves $H$ on the torus.
Thus, on $D(A_0)$, one may equivalently write
\[
A_0 v = \nu \Delta v,
\qquad v\in D(A_0).
\]
We nevertheless keep the notation $A_0=\nu P\Delta$ in order to remain consistent
with the standard Stokes operator formalism.

Since the Laplacian has no nontrivial kernel on $H$, the operator $A_0$ is
self-adjoint and strictly negative on $H$, and it generates an analytic
contraction semigroup $(e^{tA_0})_{t\ge0}$.

\subsection{Localized damping as a bounded operator}

Let $\chi_\omega\in L^\infty(\mathbb{T}^3)$ be a nonnegative localization
function.
For any scalar gain $k\ge0$, define the bounded linear operator $D_k:H\to H$ by
\[
D_k v := P(\chi_\omega\,k\,v), \qquad v\in H.
\]
Since $P$ is an orthogonal projector on $L^2$ and $\chi_\omega\in L^\infty$, one
has
\[
\|D_k\|_{\mathcal L(H)} \le \|\chi_\omega\|_{L^\infty}\,k.
\]
Moreover, for any $v\in H$,
\begin{equation}
\label{eq:Dk_positive}
\langle D_k v, v\rangle_H
=
\langle P(\chi_\omega kv),v\rangle_H
=
\langle \chi_\omega kv,v\rangle_{L^2}
=
k\int_{\mathbb{T}^3}\chi_\omega |v|^2\,dx \ge 0,
\end{equation}
which shows that $-D_k$ is dissipative on $H$.

\subsection{Frozen generator and dissipativity}

Let $\mathcal P$ be an admissible parameter set and
$K:\mathcal P\to\mathbb{R}_+$ a bounded nonnegative scalar map.
For each fixed $\rho\in\mathcal P$, define
\[
D_{K(\rho)} := D_k\big|_{k=K(\rho)},\qquad
D_{K(\rho)}v=P(\chi_\omega K(\rho)v),
\]
and introduce the frozen operator
\begin{equation}
\label{eq:Afrozen}
A(\rho):=A_0-D_{K(\rho)},
\qquad
D(A(\rho)):=D(A_0).
\end{equation}

\begin{proposition}[Generation and uniform dissipativity]
\label{prop:gen_dissip}
For every $\rho\in\mathcal P$, the operator $A(\rho)$ generates a
$\mathcal C_0$-semigroup $(e^{tA(\rho)})_{t\ge0}$ on $H$ and satisfies
\[
\|e^{tA(\rho)}\|_{\mathcal L(H)}\le1,\qquad \forall t\ge0.
\]
If there exists $\alpha_\omega>0$ such that
\begin{equation}
\label{eq:coercivity_operator}
-\langle A(\rho)v,v\rangle_H \ge \alpha_\omega\|v\|_H^2,
\qquad \forall v\in D(A_0),
\end{equation}
uniformly for $\rho\in\mathcal P$, then
\[
\|e^{tA(\rho)}\|_{\mathcal L(H)}\le e^{-\alpha_\omega t},
\qquad \forall t\ge0,
\]
uniformly in $\rho$.
\end{proposition}

\begin{proof}
Fix $\rho\in\mathcal P$.
Since $A_0$ generates an analytic contraction semigroup on $H$ and
$-D_{K(\rho)}\in\mathcal L(H)$ is bounded, the bounded perturbation theorem
implies that $A(\rho)=A_0-D_{K(\rho)}$ generates a $\mathcal C_0$-semigroup on
$H$.

For $v\in D(A_0)$,
\[
\langle A_0v,v\rangle_H
=
\nu\langle P\Delta v,v\rangle_H
=
\nu\langle \Delta v,v\rangle_{L^2}
=
-\nu\|\nabla v\|_{L^2}^2
\le 0,
\]
where we used the orthogonality of $P$ and periodic integration by parts.
Combining this with \eqref{eq:Dk_positive} yields
\[
\langle A(\rho)v,v\rangle_H
=
\langle A_0v,v\rangle_H-\langle D_{K(\rho)}v,v\rangle_H
\le 0.
\]
Thus $A(\rho)$ is dissipative on $D(A_0)$.

To prove contractivity, let $u_0\in D(A_0)$ and set
$u(t)=e^{tA(\rho)}u_0$.
Then $u$ is a strong solution of $\dot u=A(\rho)u$, and therefore
\[
\frac12\frac{d}{dt}\|u(t)\|_H^2
=
\langle A(\rho)u(t),u(t)\rangle_H
\le 0.
\]
Hence $\|u(t)\|_H\le \|u_0\|_H$ for all $t\ge0$.
By density of $D(A_0)$ in $H$ and strong continuity of the semigroup, this
estimate extends to all $u_0\in H$, which gives
\[
\|e^{tA(\rho)}\|_{\mathcal L(H)}\le 1.
\]

Assume now that \eqref{eq:coercivity_operator} holds uniformly on $\mathcal P$.
For $u_0\in D(A_0)$, the corresponding strong solution satisfies
\[
\frac12\frac{d}{dt}\|u(t)\|_H^2
=
\langle A(\rho)u(t),u(t)\rangle_H
\le -\alpha_\omega\|u(t)\|_H^2.
\]
Gr\"onwall's inequality yields
\[
\|u(t)\|_H\le e^{-\alpha_\omega t}\|u_0\|_H.
\]
Again, density and strong continuity extend the estimate to all $u_0\in H$.
\end{proof}

\begin{remark}[Origin of the nonautonomous structure]
The nonautonomous character of the linear problem comes exclusively from the
time dependence of the damping coefficient $K(\rho(t))$.
The underlying Stokes dynamics remains autonomous.
The operator formulation is introduced in order to quantify how this
time-dependent localized dissipation affects the associated energy decay
estimates.
\end{remark}

\subsection{Nonautonomous dynamics and evolution families}

Let $\rho:\mathbb R_+\to\mathcal P$ be a measurable scheduling signal and
consider the nonautonomous linear system
\begin{equation}
\label{eq:linear_nonauto}
\dot u(t)=A(\rho(t))u(t).
\end{equation}
Here $u(t)$ denotes an abstract state in $H$ and should not be confused with the
velocity field $u(\cdot,t)$ in the PDE formulation.
Solutions are described by an evolution family $(U(t,s))_{t\ge s\ge0}$.

\begin{assumption}[Regularity of the scheduling]
\label{ass:rho_reg}
For every $T>0$, the map $t\mapsto K(\rho(t))$ belongs to
$W^{1,\infty}(0,T)$ and satisfies
\[
0\le K(\rho(t))\le K_{\max}
\qquad \text{for almost every } t\in(0,T).
\]
This assumption is used only to ensure well-posedness of the nonautonomous
linear problem.
\end{assumption}

\begin{proposition}[Evolution family and decay]
\label{prop:evol_family}
Under Assumption~\ref{ass:rho_reg}, system \eqref{eq:linear_nonauto} admits a
unique evolution family $(U(t,s))_{t\ge s\ge0}$ on $H$.
Moreover,
\[
\|U(t,s)\|_{\mathcal L(H)}\le1,\qquad \forall t\ge s\ge0.
\]
If \eqref{eq:coercivity_operator} holds uniformly on $\mathcal P$, then
\[
\|U(t,s)\|_{\mathcal L(H)}\le e^{-\alpha_\omega(t-s)},
\qquad \forall t\ge s\ge0.
\]
\end{proposition}

\begin{proof}
Set
\[
A(t):=A(\rho(t))=A_0+B(t),
\qquad
B(t):=-D_{K(\rho(t))}.
\]
Since the map $k\mapsto D_k$ is linear and bounded from $\mathbb R_+$ to
$\mathcal L(H)$, Assumption~\ref{ass:rho_reg} implies that
$t\mapsto B(t)\in W^{1,\infty}(0,T;\mathcal L(H))$ for every $T>0$.
In particular, $B(\cdot)$ is strongly measurable and essentially bounded on
every bounded time interval, while the domain remains constant:
\[
D(A(t))=D(A_0),\qquad \forall t\ge0.
\]

Standard results on nonautonomous evolution equations with bounded
time-dependent perturbations of a generator with constant domain
(Kato--Tanabe theory) imply that \eqref{eq:linear_nonauto} admits a unique
evolution family $(U(t,s))_{t\ge s\ge0}$ on $H$.

Let now $u_s\in D(A_0)$ and define $u(t):=U(t,s)u_s$.
Then $u$ is a strong solution on every bounded interval and satisfies
\[
\frac12\frac{d}{dt}\|u(t)\|_H^2
=
\langle A(t)u(t),u(t)\rangle_H.
\]
Since
\[
\langle A_0v,v\rangle_H=-\nu\|\nabla v\|_{L^2}^2\le0,
\qquad
\langle D_{K(\rho(t))}v,v\rangle_H\ge0,
\]
one obtains
\[
\frac12\frac{d}{dt}\|u(t)\|_H^2\le0,
\]
hence
\[
\|u(t)\|_H\le \|u_s\|_H,\qquad \forall t\ge s.
\]
By density of $D(A_0)$ in $H$, the estimate extends to all $u_s\in H$ and
yields
\[
\|U(t,s)\|_{\mathcal L(H)}\le1.
\]

If \eqref{eq:coercivity_operator} holds uniformly on $\mathcal P$, then for
every strong solution
\[
\frac12\frac{d}{dt}\|u(t)\|_H^2
=
\langle A(t)u(t),u(t)\rangle_H
\le -\alpha_\omega\|u(t)\|_H^2.
\]
Gr\"onwall's inequality gives
\[
\|u(t)\|_H\le e^{-\alpha_\omega(t-s)}\|u_s\|_H,
\qquad \forall t\ge s.
\]
Once again, density extends the estimate to all $u_s\in H$, which proves the
claimed exponential bound.
\end{proof}

\subsection{Connection with the full Navier--Stokes system}

This subsection is interpretative and connects the linear operator analysis with
the energy estimates for the full nonlinear dynamics.

The Navier--Stokes system \eqref{eq:NS_controlled} can be written abstractly as
\[
\dot u(t)=A(\rho(t))u(t)+\mathcal B(u(t),u(t)),
\]
where
\[
\mathcal B(u,u):=-P((u\cdot\nabla)u).
\]
The skew-symmetry of $\mathcal B$ implies that it does not contribute to the
kinetic energy balance.
As a consequence, the dissipativity and exponential decay properties established
for the linear part reappear in the nonlinear system at the level of energy
estimates for Leray--Hopf weak solutions, under the localized coercivity
condition.

\begin{remark}[Role of the operator formulation]
The operator-theoretic formulation is not intended to provide a nonlinear
semigroup theory for the full Navier--Stokes equations.
Its purpose is to isolate and analyze the time-dependent dissipative structure
induced by localized damping, and to connect energy estimates for weak solutions
with stability properties of parameter-dependent linear operators.
\end{remark}

The difficulty of verifying coercivity in infinite dimension motivates the
finite-dimensional spectral procedure introduced in the next section.

\section{Continuous Observability and Localized Coercivity}
\label{sec:continuous_observability}

This section clarifies the analytical nature of the localized coercivity
condition introduced in Proposition~\ref{prop:coercivity_optimal} by relating it
to observability and spectral inequalities for parabolic operators.
The aim is not to derive new observability results for the
three-dimensional Navier--Stokes equations, but to place the coercivity condition
used throughout the paper within a precise functional-analytic setting.
This perspective also motivates the discrete spectral verification procedure
introduced in Section~\ref{sec:discrete_coercivity}.

The localized coercivity condition is not expected to hold for arbitrary damping
regions.
Its validity depends on geometric and spectral properties of the damping set
$\omega$, and it is therefore treated here through structural conditions.

\subsection{Observability for the damped Stokes dynamics}
The present analysis does not extend observability results to the nonlinear
Navier--Stokes system, but instead uses the linear theory as a structural
guideline for formulating coercivity conditions at the energy level.
Consider the linear Stokes system with localized damping
\begin{equation}
\label{eq:stokes_damped}
\partial_t v - \nu \Delta v + \nabla p + K_{\min}\chi_\omega v = 0,
\qquad \nabla\cdot v = 0,
\end{equation}
posed on the periodic domain $\mathbb{T}^3$ with divergence-free, zero-mean
initial data $v_0\in H$.
Equivalently, \eqref{eq:stokes_damped} can be written in projected form as
\[
\partial_t v = (A_0-D_{K_{\min}})v,
\]
where $A_0=\nu P\Delta$ is the Stokes operator and
$D_{K_{\min}}=P(K_{\min}\chi_\omega\,\cdot)$ is the localized damping operator.

For strong solutions generated by the associated semigroup, the energy identity reads
\begin{equation}
\label{eq:stokes_energy}
\frac{1}{2}\frac{d}{dt}\|v(t)\|_{L^2}^2
+
\nu\|\nabla v(t)\|_{L^2}^2
+
K_{\min}\|\chi_\omega^{1/2}v(t)\|_{L^2}^2
=0.
\end{equation}

\begin{definition}[Continuous observability inequality]
\label{def:observability}
System \eqref{eq:stokes_damped} is said to satisfy a continuous observability
inequality if there exists a constant $C_\omega>0$ such that
\begin{equation}
\label{eq:observability}
\|v_0\|_{L^2}^2
\le
C_\omega
\int_0^{+\infty}
\|\chi_\omega^{1/2}v(t)\|_{L^2}^2\,dt,
\qquad \forall v_0\in H,
\end{equation}
where $v$ denotes the solution of \eqref{eq:stokes_damped}.
\end{definition}

\subsection{A structural spectral condition}

The validity of coercivity or observability properties depends on the geometry of
the damping region $\omega$.
We introduce a structural assumption in the form of a spectral-type inequality.

\begin{assumption}[Spectral-type inequality]
\label{ass:geo}
The set $\omega\subset\mathbb{T}^3$ has positive measure and satisfies the
following property: there exists a constant $C_\omega>0$ such that, for all
divergence-free trigonometric polynomials $v\in H$,
\begin{equation}
\label{eq:spectral_ineq}
\|v\|_{L^2}^2
\le
C_\omega
\left(
\|\chi_\omega^{1/2}v\|_{L^2}^2
+
\|\nabla v\|_{L^2}^2
\right).
\end{equation}
\end{assumption}

This inequality should be understood as a structural spectral condition tailored
to the present coercivity problem.
It is consistent with the general philosophy of observability and spectral
inequalities of Lebeau--Robbiano type; see, for instance,
\cite{LebeauRobbiano1995,apraiz2014observability}.
Typical sufficient conditions involve thickness or non-degeneracy properties of
the set $\omega$.

\subsection{Localized coercivity and exponential stability}

The following result establishes the link between localized coercivity and
exponential decay for the damped Stokes dynamics.

\begin{theorem}[Coercivity and exponential stability]
\label{thm:observability_coercivity}
Let $\omega\subset\mathbb{T}^3$ be measurable and let $K_{\min}\ge 0$.
Assume that there exists $\alpha_\omega>0$ such that
\begin{equation}
\label{eq:coercivity_obs}
\nu\|\nabla v\|_{L^2}^2
+
K_{\min}\|\chi_\omega^{1/2}v\|_{L^2}^2
\ge
\alpha_\omega\|v\|_{L^2}^2,
\qquad \forall v\in V.
\end{equation}
Then the semigroup generated by $A_0-D_{K_{\min}}$ on $H$ is exponentially
stable in $L^2$, namely
\[
\|v(t)\|_{L^2}\le e^{-\alpha_\omega t}\|v_0\|_{L^2},
\qquad \forall t\ge0.
\]
\end{theorem}

\begin{proof}
Let $v_0\in D(A_0)$ and define
\[
v(t)=e^{t(A_0-D_{K_{\min}})}v_0.
\]
Since $A_0$ generates an analytic semigroup on $H$ and
$-D_{K_{\min}}\in\mathcal L(H)$ is a bounded perturbation, the operator
$A_0-D_{K_{\min}}$ generates a $\mathcal C_0$-semigroup on $H$.
Moreover, for $v_0\in D(A_0)$, the function $v$ is a strong solution of
\eqref{eq:stokes_damped}, and
\[
v\in C([0,\infty);D(A_0))\cap C^1([0,\infty);H).
\]

Taking the $L^2$ inner product of \eqref{eq:stokes_damped} with $v(t)$ yields
the energy identity
\begin{equation}
\label{eq:proof_energy_identity}
\frac12\frac{d}{dt}\|v(t)\|_{L^2}^2
+
\nu\|\nabla v(t)\|_{L^2}^2
+
K_{\min}\|\chi_\omega^{1/2}v(t)\|_{L^2}^2
=0,
\qquad \forall t\ge0.
\end{equation}

Applying the coercivity estimate \eqref{eq:coercivity_obs} to $v(t)\in V$, we obtain
\[
\nu\|\nabla v(t)\|_{L^2}^2
+
K_{\min}\|\chi_\omega^{1/2}v(t)\|_{L^2}^2
\ge
\alpha_\omega\|v(t)\|_{L^2}^2,
\qquad \forall t\ge0.
\]
Substituting into \eqref{eq:proof_energy_identity} gives
\[
\frac12\frac{d}{dt}\|v(t)\|_{L^2}^2
+
\alpha_\omega\|v(t)\|_{L^2}^2
\le 0.
\]
Setting $E(t)=\|v(t)\|_{L^2}^2$, we obtain
\[
\frac{d}{dt}E(t)+2\alpha_\omega E(t)\le0.
\]
By Gr\"onwall's inequality,
\[
E(t)\le e^{-2\alpha_\omega t}E(0),
\qquad \forall t\ge0,
\]
which yields
\[
\|v(t)\|_{L^2}
\le
e^{-\alpha_\omega t}\|v_0\|_{L^2},
\qquad \forall t\ge0.
\]

It remains to extend the estimate to arbitrary initial data $v_0\in H$.
Let $(v_0^n)\subset D(A_0)$ be such that $v_0^n\to v_0$ in $H$.
Then
\[
\|e^{t(A_0-D_{K_{\min}})}v_0^n\|_{L^2}
\le
e^{-\alpha_\omega t}\|v_0^n\|_{L^2}.
\]
Passing to the limit as $n\to\infty$ and using the strong continuity of the
semigroup on $H$, we obtain
\[
\|e^{t(A_0-D_{K_{\min}})}v_0\|_{L^2}
\le
e^{-\alpha_\omega t}\|v_0\|_{L^2},
\qquad \forall t\ge0.
\]
This concludes the proof.
\end{proof}

\begin{corollary}[Localized coercivity under Assumption~\ref{ass:geo}]
\label{cor:localized_coercivity_geo}
Under Assumption~\ref{ass:geo}, there exists a constant $\alpha_\omega>0$ such that
\begin{equation}
\label{eq:coercivity_geo}
\nu\|\nabla v\|_{L^2}^2
+
K_{\min}\|\chi_\omega^{1/2}v\|_{L^2}^2
\ge
\alpha_\omega \|v\|_{L^2}^2,
\qquad \forall v\in V.
\end{equation}
More precisely, one may take
\[
\alpha_\omega = C_\omega^{-1}\min(\nu,K_{\min}),
\]
where $C_\omega$ is the constant in \eqref{eq:spectral_ineq}.
\end{corollary}

\begin{proof}
By Assumption~\ref{ass:geo}, one has
\[
\|v\|_{L^2}^2
\le
C_\omega
\left(
\|\chi_\omega^{1/2}v\|_{L^2}^2
+
\|\nabla v\|_{L^2}^2
\right)
\]
for all divergence-free trigonometric polynomials $v$.
Since such functions are dense in $V$, the inequality extends to all
$v\in V$ by continuity.

Moreover,
\[
\nu\|\nabla v\|_{L^2}^2
+
K_{\min}\|\chi_\omega^{1/2}v\|_{L^2}^2
\ge
\min(\nu,K_{\min})
\left(
\|\nabla v\|_{L^2}^2
+
\|\chi_\omega^{1/2}v\|_{L^2}^2
\right).
\]
Combining the two inequalities yields \eqref{eq:coercivity_geo}.
\end{proof}

\begin{remark}
The spectral-type inequality \eqref{eq:spectral_ineq} is known to hold under
additional geometric assumptions on the set $\omega$, such as thickness or
non-degeneracy conditions.
In periodic domains, Fourier-based arguments may simplify the analysis compared
to bounded domains with boundary conditions, although the precise characterization
of admissible sets $\omega$ remains delicate.
A complete treatment of such geometric conditions is beyond the scope of the
present work.
\end{remark}

\subsection{Limitations in the nonlinear Navier--Stokes setting}

The above results are linear and rely on properties of the Stokes operator.
Quantitative observability and unique continuation estimates are well understood
for linear parabolic systems.

In contrast, extending such properties to the fully nonlinear
three-dimensional Navier--Stokes equations in a form compatible with
Leray--Hopf weak solutions remains out of reach.
No global observability inequality at the energy level, with uniform constants
and without additional regularity assumptions, is currently available.

\subsection{Motivation for the discrete spectral procedure}

In view of the above discussion, establishing localized coercivity in infinite
dimension is closely related to proving suitable spectral or observability-type
inequalities.

Since such properties depend on the geometry of $\omega$ and are difficult to
verify in practice, a finite-dimensional spectral verification procedure is
introduced in Section~\ref{sec:discrete_coercivity}.

The discrete coercivity constant $\alpha_{\omega,N}$ should be interpreted as a
computable proxy for the continuous spectral gap $\alpha_\omega$, evaluated on
finite-dimensional divergence-free subspaces.

\section{Discrete Verification of the Localized Coercivity Constant}
\label{sec:discrete_coercivity}

This section introduces a finite-dimensional spectral procedure to assess the
effectiveness of localized damping in a computable manner, consistently with the
energy-based coercivity mechanism identified at the continuous level.

\subsection{Motivation and reduced-order setting}

The exponential decay estimates established in
Section~\ref{sec:energy_stabilization} rely on a localized coercivity mechanism,
expressed either in variational form
(Proposition~\ref{prop:coercivity_optimal}) or through the operator inequality
\eqref{eq:coercivity_operator}.
In infinite dimensions, verifying such properties typically requires spectral or
observability arguments, as discussed in
Section~\ref{sec:continuous_observability}.

A complementary approach is adopted here: localized coercivity is assessed on a
finite-dimensional divergence-free subspace.
The resulting discrete constant provides a stability indicator consistent with
the continuous coercivity mechanism, restricted to a resolved modal subspace.

Let $\{\varphi_i\}_{i=1}^N \subset V$ be a divergence-free Galerkin basis of $H$,
orthonormal in $H$, that is,
\[
\langle \varphi_i,\varphi_j\rangle_H = \delta_{ij}.
\]
Define
\[
H_N := \mathrm{span}\{\varphi_1,\ldots,\varphi_N\}, \qquad
v_N = \sum_{i=1}^N z_i \varphi_i,\quad z\in\mathbb{R}^N.
\]
Then $\|v_N\|_H^2 = \|z\|^2$.

Define the matrices
\begin{align}
S_{ij} &:= \nu \langle \nabla \varphi_i, \nabla \varphi_j\rangle_{L^2},\\
(M_\omega)_{ij} &:= \langle \chi_\omega \varphi_i, \varphi_j\rangle_{L^2}.
\end{align}

\begin{lemma}
The matrices $S$ and $M_\omega$ are symmetric and positive semidefinite.
Moreover, $S$ is positive definite on $H_N$ provided that $H_N$ contains no
nontrivial constant vector fields.
\end{lemma}

\begin{proof}
Symmetry follows from the symmetry of the $L^2$ inner product.

Let $z\in\mathbb{R}^N$ and define $v_N=\sum z_i\varphi_i \in H_N\subset V$.
Then
\[
z^\top S z = \nu \|\nabla v_N\|_{L^2}^2 \ge 0,
\qquad
z^\top M_\omega z = \|\chi_\omega^{1/2} v_N\|_{L^2}^2 \ge 0,
\]
which proves positive semidefiniteness.

If $z^\top S z=0$, then $\nabla v_N=0$, hence $v_N$ is constant.
Since $v_N\in H$ has zero mean, it follows that $v_N=0$, hence $z=0$.
\end{proof}

\subsection{Discrete coercivity constant}
The discrete coercivity constant admits the following characterization.
\begin{proposition}[Discrete coercivity]
\label{prop:discrete_coercivity}
There exists a constant $\alpha_{\omega,N}\ge 0$ such that
\begin{equation}
\label{eq:coercivity_discrete}
z^\top\big(S + K_{\min} M_\omega\big) z \ge
\alpha_{\omega,N}\, \|z\|^2,
\qquad \forall z\in\mathbb{R}^N,
\end{equation}
and the largest such constant is
\[
\alpha_{\omega,N} = \lambda_{\min}(S + K_{\min}M_\omega).
\]
\end{proposition}

\begin{proof}
The matrix $A:=S+K_{\min}M_\omega$ is symmetric.
By the spectral theorem, it admits an orthonormal eigenbasis with real eigenvalues.
Let $\lambda_{\min}$ denote its smallest eigenvalue.

For any $z\in\mathbb{R}^N$, writing $z$ in this eigenbasis yields
\[
z^\top A z \ge \lambda_{\min} \|z\|^2.
\]
Optimality follows since equality holds for eigenvectors associated with
$\lambda_{\min}$.
\end{proof}

\subsection{Reduced-order dynamics and well-posedness}

Projecting \eqref{eq:NS_controlled} onto $H_N$ yields
\begin{equation}
\label{eq:ROM_closed_loop}
\dot z(t) = A_N(\rho(t)) z(t) + \mathcal{N}_N(z(t)),
\end{equation}
where $A_N(\rho)$ is the reduced linear operator and $\mathcal{N}_N$ is the
reduced convection operator.

\begin{lemma}[Well-posedness]
\label{lem:ROM_wellposed}
For any $z(0)\in\mathbb{R}^N$, system \eqref{eq:ROM_closed_loop} admits a unique
global solution $z\in C^1([0,\infty);\mathbb{R}^N)$.
\end{lemma}

\begin{proof}
The right-hand side of \eqref{eq:ROM_closed_loop} is a polynomial function of $z$
and is therefore locally Lipschitz.
Hence a unique maximal solution exists.

Assume that the nonlinear term is energy-preserving:
\[
\langle \mathcal{N}_N(z), z\rangle = 0.
\]
Then
\[
\frac{1}{2}\frac{d}{dt}\|z(t)\|^2
=
- z(t)^\top\big(S + K(\rho(t))M_\omega\big)z(t)
\le 0.
\]
Thus $\|z(t)\|$ is non-increasing and remains bounded for all time.
Therefore no finite-time blow-up can occur, and the solution is global.
\end{proof}

\subsection{Exponential decay of the reduced model}
Exponential decay follows from the localized coercivity property.
\begin{proposition}[Exponential decay]
\label{prop:ROM_decay}
Assume that $K(\rho(t))\ge K_{\min}>0$.
Then any solution of \eqref{eq:ROM_closed_loop} satisfies
\[
\|z(t)\| \le e^{-\alpha_{\omega,N} t}\|z(0)\|.
\]
\end{proposition}

\begin{proof}
Taking the Euclidean inner product of \eqref{eq:ROM_closed_loop} with $z(t)$ yields
\[
\frac{1}{2}\frac{d}{dt}\|z(t)\|^2
=
- z(t)^\top\big(S + K(\rho(t))M_\omega\big)z(t).
\]
Using $K(\rho(t))\ge K_{\min}$ and
Proposition~\ref{prop:discrete_coercivity}, one obtains
\[
z^\top(S + K(\rho(t))M_\omega)z
\ge
\alpha_{\omega,N}\|z\|^2.
\]
Hence
\[
\frac{d}{dt}\|z(t)\|^2 \le -2\alpha_{\omega,N}\|z(t)\|^2.
\]
Gr\"onwall's inequality yields the result.
\end{proof}

\subsection{Monotonicity and robustness}
The discrete coercivity constant satisfies the following monotonicity and robustness properties.
\begin{proposition}[Monotonicity]
\label{prop:alpha_monotone}
The map $K_{\min}\mapsto \alpha_{\omega,N}(K_{\min})$ is nondecreasing.
\end{proposition}

\begin{proof}
Let $K_2\ge K_1$. Then
\[
S+K_2M_\omega = (S+K_1M_\omega) + (K_2-K_1)M_\omega,
\]
with $(K_2-K_1)M_\omega \succeq 0$.
The monotonicity of the smallest eigenvalue follows from the min--max principle.
\end{proof}

\begin{proposition}[Robustness under scheduling]
\label{prop:alpha_lpv}
If $K(\rho(t))\ge K_{\min}$ almost everywhere, then the decay rate is bounded
below by $\alpha_{\omega,N}(K_{\min})$.
\end{proposition}

\begin{proof}
The proof follows from the same energy estimate as in
Proposition~\ref{prop:ROM_decay}.
\end{proof}

\subsection{Discussion}
The constant $\alpha_{\omega,N}$ yields a computable lower bound for
dissipation in the reduced system.
It yields a rigorous exponential decay estimate for the reduced-order model and
is consistent with the continuous coercivity mechanism identified in
Section~\ref{sec:energy_stabilization}.

No convergence toward an infinite-dimensional coercivity constant is claimed.
The role of $\alpha_{\omega,N}$ is to provide a sufficient condition for decay at
the level of a chosen finite-dimensional approximation.

\section{Conclusion}

An energy-based analysis has been developed for the three-dimensional
incompressible Navier--Stokes equations with spatially localized volumic
dissipation.
The results are formulated at the level of kinetic energy and remain compatible
with Leray--Hopf weak solutions, without relying on global regularity or
linearization arguments.

The core of the analysis is the identification of a \emph{localized coercivity
property} linking viscous diffusion and spatially restricted damping to global
$L^2$ control of the kinetic energy.
This property is formulated through variational inequalities and
operator-theoretic estimates for a family of parameter-dependent Stokes-type
operators, and provides a sufficient condition for exponential energy decay in a
nonautonomous setting.

At the linear level, coercivity is related to observability and spectral
inequalities for parabolic operators, providing an interpretation of the
coercivity condition in terms of a spectral gap induced by localized
dissipation under suitable geometric conditions.
In the nonlinear Navier--Stokes setting, these properties are used only through
energy estimates and do not rely on nonlinear observability arguments.

A finite-dimensional spectral verification procedure has been introduced to
complement the infinite-dimensional analysis.
The resulting discrete coercivity constant yields a computable lower bound on
the decay rate for divergence-free reduced-order models and preserves key
structural features such as monotonicity with respect to the damping strength
and robustness under time-dependent gain variations.
This construction is not intended as a convergence result, but rather as a
finite-dimensional certification tool consistent with the continuous coercivity
mechanism.

Overall, the results provide a coherent framework in which localized dissipation
can be quantified through the interplay between energy estimates,
operator-theoretic arguments for time-dependent systems, and
observability-type coercivity properties, in a setting adapted to
three-dimensional weak solutions.

\paragraph{Perspectives}
A first direction is to characterize more precisely the geometric conditions on
the damping region that ensure positivity of the localized coercivity constant,
in connection with spectral and observability inequalities.

A second direction concerns the extension of the analysis to more general
configurations, including boundary damping mechanisms and coupled fluid systems.

Another perspective is the study of flows around nontrivial equilibria, where
instabilities may arise and the role of localized dissipation should be analyzed
in connection with linearized (Oseen-type) operators.

Finally, the relation between continuous coercivity properties and reduced-order
spectral quantities deserves further investigation, with the aim of providing
reliable computable indicators for energy-dissipative mechanisms in
high-dimensional fluid models
\section*{Funding}
This research did not receive funding.










\end{document}